\documentclass[12pt]{amsart}
\usepackage[english]{babel}
\usepackage{amsmath}
\usepackage{amsthm}
\usepackage{amssymb}
\usepackage{mathrsfs}
\usepackage{enumerate}
\usepackage[notcite, final, notref]{showkeys}
\usepackage{dsfont}

\usepackage[dvips]{color}
\newtheorem{theorem}{Theorem}[section]
\newtheorem{problem}[theorem]{Problem}

\newtheorem{proposition}[theorem]{Proposition}

\newtheorem{definition}[theorem]{Definition}

\theoremstyle{definition}
\newtheorem{remark}[theorem]{Remark}

\usepackage{xcolor}

\title[Fock space as a de Branges-Rovnyak space]{The Fock space as a de Branges-Rovnyak space}

\author[D. Alpay]{Daniel Alpay}
\address{(DA)
Faculty of Mathematics, Physics, and Computation\\
Schmid College of Science and Technology\\
Chapman University\\
One University Drive
Orange, California 92866\\
USA}
\email{alpay@chapman.edu}
\thanks{Daniel Alpay thanks the Foster G. and Mary McGaw Professorship in
Mathematical Sciences, which supported this research.}

\author[F. Colombo]{Fabrizio Colombo}
\address{(FC) Politecnico di
Milano\\Dipartimento di Matematica\\Via E. Bonardi, 9\\20133 Milano,
Italy}
\email{fabrizio.colombo@polimi.it}

\author[I. Sabadini]{Irene Sabadini}
\address{(IS) Politecnico di
Milano\\Dipartimento di Matematica\\Via E. Bonardi, 9\\20133 Milano\\Italy}
\email{irene.sabadini@polimi.it}

\begin{document}
\maketitle
\begin{abstract}
We show that de Branges-Rovnyak spaces include as special cases a number of spaces, such as the Hardy space, the
Fock space, the Hardy-Sobolev space and the Dirichlet space. We present a general framework in which all these spaces can be obtained by specializing a sequence that appears in the construction.
We show how to exploit this approach to solve interpolation problems in the Fock space.
\end{abstract}


\noindent{\em }
\date{today}

\section{Introduction}
\setcounter{equation}{0}
Let $S$ be a function on the open unit disc $\mathbb D$ of the complex plane which is analytic and contractive and consider the kernel
\begin{equation}
\label{kernelKSscalar}
K_S(z,w)=\frac{1-S(z)S(w)^*}{1-z\overline{w}}.
\end{equation}
The reproducing kernel Hilbert space with kernel $K_S$ was introduced by de Branges and Rovnyak in \cite{dbr1,dbr2} and it is nowadays known as de Branges-Rovnyak space and is denoted by $\mathcal H(K_S)$ (in the literature it is also denoted by $\mathcal H(S)$).
 These type of spaces can be defined also when the kernel $K_S$ is defined by
\begin{equation}
\label{kernelKS}
K_S(z,w)=\frac{I_{\mathcal C}-S(z)S(w)^*}{1-z\overline{w}}
\end{equation}
 where the function $S$ is operator valued and  analytic in the open unit disk. They form an important family of reproducing kernel spaces, in which several problems can be naturally formulated and solved.
\\
In the operator valued case, we consider two Hilbert spaces $\mathcal U$, $\mathcal C$ and the space $\mathbf{L}(\mathcal  U, \mathcal  C)$
 of bounded, linear operators from $\mathcal  U$ to $\mathcal  C$. The set of functions $S:\ \mathbb D \to \mathbf{L}(\mathcal  U, \mathcal  C)$ analytic in $\mathbb D$ whose values are contractive operators is denoted by $\mathfrak S(\mathcal U,\mathcal C)$ and such $S$ are called Schur functions. We note that one can further generalize by taking Pontryagin spaces instead of Hilbert spaces. \\
 There are various characterizations of Schur functions, for example in terms of the operator $M_S$ of multiplication by $S$, but the characterization which is relevant in this work is in terms of the positivity of the kernel $K_S$.
\\
To this end, let us recall that the kernel $K_S$ is positive in $\mathbb D$ if for any $N\in\mathbb N$, $z_1,\ldots, z_N\in\mathbb D$ and $c_1,\ldots, c_N\in\mathcal Y$
$$
\sum_{i,j=1}^N\langle K_S(z_i,z_j)y_j,y_i\rangle_{\mathcal C}\geq 0.
$$
In this paper we show that various spaces, such as the Hardy space, the
Fock space, the Hardy-Sobolev space and the Dirichlet space, are special cases of
de Branges-Rovnyak spaces. In fact, we have a general construction in which all these spaces can be obtained by specializing a sequence that appears in it. This unified approach to these spaces has obvious advantages. For example,  we will  show how to exploit the results  in the papers
\cite{MR2794390,bbol2011}, in which the authors formulate and solve interpolation problems  in the de Branges-Rovnyak spaces, to explicitly solve interpolation problems in the Fock space.
This is a far reaching result, and it is not clear how to obtain it using different techniques. 

Although
these spaces focus on operator-valued functions, we consider here interpolation
of scalar functions to ease the notation.

The paper contains five sections, besides this Introduction. In Section 2 we show a construction which allows to embed various spaces in the framework of de Branges-Rovnyak spaces via functions of the form \eqref{S=z}. In Section 3 we recall some known facts on interpolation in reproducing kernel Hilbert spaces and we prove a simple but useful necessary and sufficients condition to solve an interpolation problem in this context. In Section 4 we show a characterization of contractive multipliers between de Branges-Rovnyak spaces, while in Section 5 we study and solve a Nevanlinna-Pick problem in this context. Finally, in Section 6 we provide the Schur algorithm for functions of the form \eqref{S=z}.

We conclude with a notation: for $a\in\ell_2(\mathbb N_0)$ we denote by $a^*$ adjoint of the map $z\mapsto za$ defined from $\mathbb C$ into
$\ell_2(\mathbb N_0)$. It is the functional:
\begin{equation}
\label{a*}
a^*(b)=\langle b,a\rangle_{\ell_2(\mathbb N_0)},\quad b\in \ell_2(\mathbb N_0),
\end{equation}
and we will use the notation $ba^*$ rather than $a^*(b)$.
\section{A new family of $\mathcal H(K_S)$ spaces}
\setcounter{equation}{0}
In this section we illustrate how to embed the spaces mentioned in the introduction in the framework of
the de Branges-Rovnyak spaces.
The crucial fact is the following result:
\begin{theorem}
\label{f=b}
Let $\mathbf c=
(c_n)_{n\in\mathbb N_0}$ be an non-decreasing sequence of numbers with $c_0=1$, and let
\begin{equation}
K_{\mathbf c}(z,w)=\sum_{n=0}^\infty\frac{z^n\overline{w}^n}{c_n}.
\end{equation}
Then there exists a $\ell_2(\mathbb N_0)$-valued entire function $S(z)$ such that
\begin{equation}
\label{KcHs}
K_{\mathbf c}(z,w)=\frac{1-\langle S(z),S(w)\rangle_{\ell_2(\mathbb N_0)}}{1-z\overline{w}}.
\end{equation}
\end{theorem}

\begin{proof}  We observe that we  can write
\[
\begin{split}
\sum_{n=0}^\infty\frac{z^n\overline{w}^n}{c_n}&=\frac{1}{1-z\overline{w}}\left(\sum_{n=0}^\infty\frac{z^n\overline{w}^n}{c_n}(1-z\overline{w})\right)\\
&=\frac{1}{1-z\overline{w}}\left(1-\sum_{n=1}^\infty z^n\overline{w}^n\left(\frac{1}{c_{n-1}}-\frac{1}{c_n}\right)\right).
\end{split}
\]
Setting
\begin{equation}
\sigma_n=\sqrt{\frac{1}{c_{n}}-\frac{1}{c_{n+1}}},\quad n=0,1,\ldots
\label{sigman}
\end{equation}
and performing simple calculations, it is clear that $K_{\mathbf c}$ is of the form \eqref{KcHs} with
\begin{equation}
S(z)=\begin{pmatrix}z\sigma_0&\cdots&z^n\sigma_{n-1}&\cdots&\end{pmatrix}.
\label{S=z}
\end{equation}
\end{proof}
\begin{definition}
Let $\mathbf c=
(c_n)_{n\in\mathbb N_0}$ be an non-decreasing sequence of numbers with $c_0=1$.
We denote by $\mathcal H(K_{\mathbf c})$ the reproducing kernel Hilbert space with reproducing kernel $K_{\mathbf c}$.
\end{definition}
\begin{remark}
We note that $\mathcal H(K_{\mathbf c})$ is the set of power series
\begin{equation}
\label{fhc}
f(z)=\sum_{n=0}^\infty z^nf_n
\end{equation}
endowed with the norm
\begin{equation}
\label{hcnorm}
\|f\|_{{\mathbf c}}=\left(\sum_{n=0}^\infty c_n|f_n|^2\right)^{1/2}.
\end{equation}
\end{remark}
To state the next result, we recall that the backward shift operator $R_0$
given by
\[
R_0f(z)=\begin{cases}\, \dfrac{f(z)-f(0)}{z},\quad z\not=0\in\Omega,\\
\,f^\prime(0),\quad\hspace{13mm} z=0,\end{cases}
\]
is defined for functions $f$ analytic in a neighborhood $\Omega$ of the origin.
We have:
\begin{proposition}\label{Prop2:3}
The inequality
\begin{equation}
\label{dbrineq}
\|R_0f\|_{\mathbf c}^2\le \|f\|_{\mathbf c}^2-|f(0)|^2
\end{equation}
with $\|\cdot\|_{\mathbf c}$ denoting the norm in $\mathcal H(K_{\mathbf c})$, holds for functions in $\mathcal H(K_{\mathbf c})$.
\end{proposition}

\begin{proof}
With $f$ given by \eqref{fhc} we have $R_0f(z)=\sum_{n=1}^\infty f_{n}z^{n-1}$ and so
\[
\begin{split}
\|R_0f\|^2_{\mathbf c}&=\sum_{n=1}^\infty c_{n-1}|f_n|^2\\
&\le \sum_{n=1}^\infty c_{n}|f_n|^2\\
&=\|f\|_{\mathbf c}^2-|f(0)|^2.
\end{split}
\]
\end{proof}

As a consequence of Proposition \ref{Prop2:3}, the form \eqref{KcHs} for the reproducing kernel of $\mathcal H(K_{\mathbf c})$ follows also from the general structure theorem
for $\mathcal H(K_S)$ spaces, proved in \cite[Theorem 3.1.2, p. 85]{adrs}, that we recall here for the reader's convenience.
The version below is stated in the positive case but the proof in \cite[Theorem 3.1.2, p. 85]{adrs}
is given in the setting of Pontryagin spaces.

\begin{theorem}
\label{thm312} Let $\mathcal C$ and $\mathcal U$ denote Hilbert spaces.
Let $\mathcal H$ be a reproducing kernel Hilbert space of $\mathcal C$-valued functions analytic in a
neighborhood of the origin and invariant under the backward shift operator $R_0$ and satisfying the inequality
\begin{equation}
\label{ineq123456}
\|R_0f\|^2_{\mathcal H}\le \|f\|_{\mathcal H}^2-\|f(0)\|_{\mathcal C}^2,\quad \forall\, f\,\in\,\mathcal H.
\end{equation}
Then there exist an Hilbert space $\mathcal U$ and a $\mathbf L(\mathcal U,\mathcal C)$-valued analytic function $S$ such that $\mathcal H$ is the restriction to $\Omega$ of
the space $\mathcal H(S)$ with reproducing kernel $K_S(z,w)$.

Conversely, every $\mathcal H(S)$ space is $R_0$-invariant and satisfies inequality \eqref{ineq123456}.
\end{theorem}

Some important examples of $\mathcal H(K_{\mathbf c})$ spaces are listed in the following table, in which we specialize the sequence $\mathbf c=
(c_n)_{n\in\mathbb N_0}$ and the corresponding kernel:\\

\begin{tabular}{|l|l|l|}
\hline
{\rm Space}&{\rm Coefficients}&{\rm Kernel}\\
\hline
\hbox{\rm Hardy space}&$c_n=1$&$\frac{1}{1-z\overline{w}}$\\
\hline
\hbox{\rm Dirichlet space}&$c_n=n$&$1-\ln(1-z\overline{w})$\\
\hline
\hbox{\rm Hardy-Sobolev space}&$c_n=n^2+1$&$\sum_{n=0}^\infty\frac{z^n\overline{w}^n}{n^2+1}$\\
\hline
\hbox{{\rm Fock space}}&$c_n=n!$&$e^{z\overline{w}}$\\
\hline{\rm Generalized Fock spaces }&$c_n=(n!)^p,\, p\in\mathbb N$&see \cite{alpay2018generalized}\\
{\rm related to Stirling numbers}& &\\
\hline
\hbox{Space associated to}&$c_n=\Gamma(\alpha n+\beta)/\Gamma(\beta)$&$\Gamma(b)E_{\alpha,\beta}(z\overline{w})$\\
the Mittag-Leffler function&  $\Gamma$ is the Gamma function&\\
                              &$\alpha$ and $\beta$ are $>0$&\\
\hline
\end{tabular}
\\

We note that the first three kernels in the table are defined only in the open unit disk, while the last three kernels are defined in the whole complex plane.
\begin{remark}{\rm
Not all the spaces that one would expect to fit this description can effectively be obtained in this way. For example the Bergman kernel
\[
\frac{1}{(1-z\overline{w})^2}=\sum_{n=0}^\infty (n+1)z^{n}\overline{w}^n
\]
does not fit the conditions of Theorem \ref{f=b}.}
\end{remark}
\section{Interpolation in reproducing kernel Hilbert spaces}
\setcounter{equation}{0}
It is useful at this stage to recall some known facts on interpolation
in reproducing kernel Hilbert spaces.
Let $K(z,w)$ be positive definite on the set $\Omega$, meaning that, for every $q\in\mathbb N$ and
$z_1,\ldots, z_q\in\Omega$ the $q\times q$ matrix with $(j,k)$ entry $K(z_j,z_k)$ is non-negative,
and let $\mathcal H(K)$ denote the associated reproducing kernel Hilbert space. We denote by
$C_w$ the evaluation at the point $w$, which is a linear bounded operator from $\mathcal H(K)$ into
$\mathbb C$. As it is well known, its adjoint is given by the formula
\begin{equation}
(C_w^*1)(z)=K(z,w),\quad z\in\Omega.
\label{Cstar}
\end{equation}
 Furthermore, a function $f\in\mathcal H(K)$ and has norm less than $1$ if and only if the
kernel
\begin{equation}
K(z,w)-f(z)\overline{f(w)}
\label{functionhs}
\end{equation}
is positive definite in $\Omega$.

\begin{problem} Given $M\in\mathbb N$ and $M$ pairs $(w_1,x_1),\ldots, (w_M,x_M)\in\Omega\times \mathbb C$, find a necessary and sufficient condition for $f\in\mathcal H(K)$ to exist such that
\[
f(w_m)=x_m,\quad m=1,\ldots, M\quad{\it and}\quad \|f\|_{\mathcal H(K)}\le 1,
\]
and describe the set of all solutions when this condition is in force.
\label{pb01}
\end{problem}

To give an answer to this problem, it is useful to set
\begin{equation}
\mathbf x=\begin{pmatrix}x_1\\ x_2\\ \vdots \\x_M\end{pmatrix}\in\mathbb C^M.
\label{x_inter}
\end{equation}
The following useful result might be known, but since we do not have any reference for it, we insert its proof:
\begin{proposition}\label{Prop3:2}
Let $P=(K(w_j,w_k))_{j,k=1,\ldots, m}\ge 0$. A necessary and sufficient condition for Problem \ref{pb01} to have a solution is that
\begin{equation}
\label{internecessity}
P-{\bf x}{\bf x}^*\ge 0.
\end{equation}
\end{proposition}

\begin{proof}
It follows from \eqref{functionhs} that \eqref{internecessity} is a necessary condition. Assume now that \eqref{internecessity} is in force, and let $\mathcal M$ denote the linear span
of the functions $K(\cdot, w_m),\, m=1,\ldots M$.
The linear span of the pairs $(K(\cdot, w_m),\overline{x_m})$ defines a linear relation $T$ on
$\mathcal M\times\mathbb C$. We now prove that this relation is contractive. Let $c_1,\ldots, c_M\in\mathbb C$ and $c=\begin{pmatrix}c_1&c_2&\cdots&c_M\end{pmatrix}^t\in\mathbb C^M$. We have
\[
\begin{split}
\langle \sum_{m=1}^MK(\cdot, w_m)c_m\, ,\, \sum_{j=1}^MK(\cdot, w_j)c_j\rangle_{\mathcal H(K)}&=c^*Pc\\
&\ge c^*{\bf x}{\bf x}^*c\\
&=\big|\sum_{m=1}^Mc_m\overline{x_m}\big|^2.
\end{split}
\]
So $T$ is the graph of a contraction, and its adjoint is given by $T^*1$ is in $\mathcal H(K)$ and has norm less that $1$. Moreover, $T^*1$ solves the interpolation problem since
\[
\begin{split}
(T^*1)(w_m)&=\langle T^*1,K(\cdot, w_m)\rangle_{\mathcal H(K)}\\
&=\langle 1,T(K(\cdot, w_m))\rangle_{\mathbb C}\\
&=x_m,\quad m=1,\ldots, M.
\end{split}
\]
\end{proof}

In the particular case of the Fock space, denoted by $\mathcal F$  Proposition \ref{Prop3:2} becomes:

\begin{theorem}
Let $w_1,x_1,w_2,x_2,\ldots, w_M,x_M\in\mathbb C$, and let ${\mathbf x}$ be as in \eqref{x_inter}. Then, there exists $f\in\mathcal F$ solution to Problem \ref{pb01} if and only if
\[
\left(e^{w_j\overline{w_k}}\right)_{j,k=1,\ldots, M}\ge {\mathbf x}{\mathbf x}^*.
\]
\end{theorem}

Assuming $P>0$ it is well known that the minimal norm solution is
\begin{equation}
\label{min}
f_{\rm min}(z)=
\begin{pmatrix} K(z,w_1)&K(z,w_2)& \cdots& K(z,w_M)\end{pmatrix} P^{-1}{\bf x}
\end{equation}
with norm
\[
\|f_{\rm min}\|^2_{\mathcal H(K)}={\mathbf x}^*P^{-1}\mathbf x.
\]
The problem is to describe the orthogonal complement of the span of the functions $K(z,w_m)$, $m=1,\ldots, M$.

When $P\geq 0$  the functions $K(z,w_m)$, $m=1,\ldots, M$ are linearly dependent, so we need to add compatibility conditions. Then we can assume that the first, say, $\tilde M$ functions are linearly independent and we can repeat the arguments given for $P>0$.

\section{A brief review of $\mathcal H(K_S)$ spaces and their multipliers}
\setcounter{equation}{0}
\label{secmult}

We begin by recalling an important factorization theorem, called Leech's theorem, which is used in this section
in the proof of Theorem \ref{th67}:

\begin{theorem}
Let $\mathcal H_1,\mathcal H_2$ and $\mathcal H_3$ denote three Hilbert spaces and let
$A$ and $B$ be two analytic functions, respectively $\mathbf L(\mathcal H_2,\mathcal H_1)$-
and $\mathbf L(\mathcal H_3,\mathcal H_1)$-valued, and such that the kernel
\[
\frac{A(z)A(w)^*-B(z)B(w)^*}{1-z\overline{w}}
\]
is positive definite in the open unit disk $\mathbb D$. Then, there exists an analytic $\mathbf L(\mathcal H_3,\mathcal H_2)$-valued function $S$, whose values are contractive operators, and such that
$B=AS$.
\end{theorem}

This result was originally proved by Leech in the early seventies, but it appeared much later,
see \cite{MR3147403,MR3147402}. Nowadays it is known that the result can be obtained using the so-called commutant lifting
theorem, see \cite{rr-univ}, but one can also use interpolation theory to prove it.

The approach of \cite{kky} will lead to a description of all solutions of the
factorization problem. We also note that Leech's theorem holds in the quaternionic setting, see e.g. \cite{zbMATH06658818} for a proof.\\

The next result characterizes elements in $\mathcal H(K_S)$. It is a special case of \cite[Theorem 11.1, p. 61]{ab6}, where contractive multipliers between de Branges-Rovnyak spaces are characterized.

\begin{theorem}
A function $h$ belongs to $\mathcal H (S)$ if and only if there exists
\[
\widetilde{\Sigma}(z)=\begin{pmatrix}\widetilde{\Sigma}_{11}(z)&\widetilde{\Sigma}_{12}(z)\\ \widetilde{\Sigma}_{21}(z)&\widetilde{\Sigma}_{22}(z)
\end{pmatrix}\in\mathfrak{S}(\ell_2(\mathbb N_0)\oplus\mathbb C,\ell_2(\mathbb N_0)\oplus\mathbb C).
\]
such that
\begin{equation}
S(z)= \widetilde{\Sigma}_{11}(z)+z\widetilde{\Sigma}_{12}(z)(1-z\widetilde{\Sigma}_{22}(z))^{-1}\widetilde{\Sigma}_{21}(z)\quad{with}\quad h(z)=\widetilde{\Sigma}_{12}(z)(1-z\widetilde{\Sigma}_{22}(z))^{-1}.
\label{mult678}
\end{equation}
\label{th67}
\end{theorem}

\begin{proof}
We have, where the first line follows from \eqref{functionhs} and the third one from Leech's factorization theorem (with $A(z)=\begin{pmatrix}1&zh(z)\end{pmatrix}$ and $B(z)= \begin{pmatrix}S(z)&h(z)\end{pmatrix}$):
\begin{eqnarray}
\label{111}
\ h\in\mathcal H(S)\,{\rm and}\,\|h||\le 1&\,\iff\,& \frac{1-\langle S(z),S(w)\rangle}{1-z\overline{w}}-h(z)
\overline{h(w)}\ge 0\\
&\,\iff\,&
\frac{1+zh(z)\overline{wh(w)}-\langle S(z),S(w)\rangle-h(z)\overline{h(w)}}{1-z\overline{w}}\ge0\\
&\,\iff\,&\begin{pmatrix}S(z)&h(z)\end{pmatrix}=
\begin{pmatrix}1&zh(z)\end{pmatrix}
\begin{pmatrix}\widetilde{\Sigma}_{11}(z)&\widetilde{\Sigma}_{12}(z)\\ \widetilde{\Sigma}_{21}(z)&\widetilde{\Sigma}_{22}(z)\end{pmatrix}
\label{f_in_rkhs}
\end{eqnarray}
where
\[
\widetilde{\Sigma}(z)=\begin{pmatrix}\widetilde{\Sigma}_{11}(z)&\widetilde{\Sigma}_{12}(z)\\ \widetilde{\Sigma}_{21}(z)&\widetilde{\Sigma}_{22}(z)
\end{pmatrix}\in\mathfrak{S}(\ell_2(\mathbb N_0)\oplus\mathbb C,\ell_2(\mathbb N_0)\oplus\mathbb C).
\]
So, we deduce
\[
\begin{split}
h\in\mathcal H(S)&\,\,\,\iff\,\,\,\left\{\begin{matrix}
S(z)&\,=\,& \hspace{-2mm}\widetilde{\Sigma}_{11}(z)+zh(z)\widetilde{\Sigma}_{21}(z)\\
h(z)&\,=\,&\widetilde{\Sigma}_{12}(z)+zh(z)\widetilde{\Sigma}_{22}(z)\end{matrix}\right.\\
&\,\,\,\iff\,\,\,\left\{\begin{matrix}
S(z)&\,=\,& \widetilde{\Sigma}_{11}(z)+z\widetilde{\Sigma}_{12}(z)(I-z\widetilde{\Sigma}_{22}(z))^{-1}\widetilde{\Sigma}_{21}(z)\\
h(z)&\,=\,&\hspace{-26mm}\widetilde{\Sigma}_{12}(z)(I-z\widetilde{\Sigma}_{22}(z))^{-1}\end{matrix}\right.
\end{split}
\]
and the statement follows.
\end{proof}
\begin{remark} Given two Schur functions, $S$ and $\sigma$, we can consider the case in which $h$ is
 a contractive multiplier from space $\mathcal H(K_\sigma)$ into $\mathcal H(K_S)$. In this case, equations \eqref{mult678} become (see \cite[(11.3) and (11.4), p. 61]{ab6}):
\begin{eqnarray}
\label{mult67890}
S(z)&=& \widetilde{\Sigma}_{11}(z)+\widetilde{\Sigma}_{12}(z)(I-\sigma(z)\widetilde{\Sigma}_{22}(z))^{-1}\sigma(z)\widetilde{\Sigma}_{21}(z)\\
h(z)&=&\widetilde{\Sigma}_{12}(z)(I-\sigma(z)\widetilde{\Sigma}_{22}(z))^{-1}.
\label{mult6789}
\end{eqnarray}
\end{remark}
\section{Interpolation in $\mathcal H(K_S)$ spaces}
\setcounter{equation}{0}
In this section we will illustrate on a very simple case how the results of
\cite{MR2794390,bbol2011} can be used to study interpolation in the $\mathcal H(K_{\mathbf c})$ spaces, and in particular in the Fock space and the Dirichlet space,
to mention two special cases. We focus on the case where $\mathcal U=\mathbb C$ and $\mathcal C=\ell_2(\mathbb N_0)$ in \eqref{kernelKS}, and so the corresponding space
$\mathcal H(K_S)$ space consists of complex-valued functions. Specifically, we will consider the following Nevanlinna-Pick problem:

\begin{problem}
\label{pb1}
Given a $\mathbf (\ell_2(\mathbb N),\mathbb C)$-valued Schur function $S$ and
$M$ pairs of points
\[
(w_1,x_1),\ldots, (w_M,x_M)\in\mathbb D \times \mathbb C,
\]
find a necessary and sufficient condition for $f\in\mathcal H(K_S)$ to exist such that
\[
f(w_m)=x_m,\quad m=1,\ldots, M,\quad{\it and}\quad \|f\|_{\mathcal H(S)}\le 1,
\]
and describe the set of all functions when this condition is in force.
\end{problem}

Following the notation in \cite{MR2794390}, we set ${\mathbf x}$ as in \eqref{x_inter},
\begin{equation}
E=\begin{pmatrix}1& 1& \cdots &1\end{pmatrix}\in\mathbb C^{1\times M},
\end{equation}
\begin{equation}
N^*=\begin{pmatrix}S(w_1)\\ S(w_2)\\ \vdots\\
S(w_M)\end{pmatrix}\in\mathbf L(\mathbb C^M,\mathbb C),
\quad{\rm and}\quad
T={\rm diag}\,(\overline{w_1},\ldots, \overline{w_M})\in\mathbb C^{M\times M}.
\end{equation}
Furthermore, we set
\[
\begin{split}
F_S(z)&=(E-S(z)N)(I_M-zT)^{-1}\\
&=\begin{pmatrix} K_S(z,w_1)&K_S(z,w_2)& \cdots& K_S(z,w_M)\end{pmatrix}\in(\mathcal H(K_S))^{1\times M}.
\end{split}
\]
The function $F_S$ defines a multiplication operator $M_{F_S}$ from $\mathbb C^M$ into $\mathcal H(K_S))^{1\times M}$, and we have
\begin{equation}
\begin{split}
P&\stackrel{\rm def.}{=}
M_{F_S}^*M_{F_S}\\
&=\begin{pmatrix} C_{w_1}\\ C_{w_2}\\ \vdots \\ C_{w_M}\end{pmatrix}\begin{pmatrix} K_S(z,w_1)&K_S(z,w_2)&\cdots &K_S(z,w_M)\end{pmatrix}\\
&=\begin{pmatrix} K_S(w_1,w_1)&K_S(w_1,w_2)&\cdots &K_S(w_1,w_M)\\
K_S(w_2,w_1)&\cdots & &K_S(w_2,w_M)\\
\vdots& & &\vdots\\
K(w_M,w_1)&K(w_M,w_2)&\cdots &K_S(w_M,w_M)\end{pmatrix}\ge 0.
\end{split}
\end{equation}
By \eqref{x_inter} a necessary and sufficient condition for the existence of a solution to the above Nevanlinna-Pick interpolation problem is that
\begin{equation}
\label{345}
P-\mathbf {xx^*}\ge 0.
\end{equation}
To describe the set of all solutions, the first step is to construct a rational co-isometric function
\begin{equation}
\Sigma(z)=\begin{pmatrix}\Sigma_{11}(z)&\Sigma_{12}(z)\\
\Sigma_{21}(z)&\Sigma_{22}(z)\end{pmatrix}\,:\, \begin{pmatrix}\mathbb  C\\
\mathcal D_1\end{pmatrix}\,\rightarrow\, \begin{pmatrix}\ell_2(\mathbb N_0)\\ \mathcal D_2
\end{pmatrix},
\end{equation}
where $\mathcal D_1$ and $\mathcal D_2$ are Hilbert spaces, and with realization along these spaces decompositions given by
\begin{equation}
\label{real123456}
\Sigma(z)=\begin{pmatrix}D_{11}&D_{12}\\ D_{21}&0\end{pmatrix}+z\begin{pmatrix}C_1\\ C_2\end{pmatrix}\left(I-zA\right)^{-1}\begin{pmatrix}B_1&B_2\end{pmatrix},\quad z\in\mathbb D.
\end{equation}
We will not review the construction of \cite{MR2794390}, but limit ourselves to
mention that  $\Sigma$ is the transfer function
of a unitary colligation obtained from the identity (see
\cite[p. 236]{MR2794390})
\begin{equation}
\label{iso123123}
\|\sqrt{P}x\|^2+\|Nx\|^2=\|\sqrt{P}Tx\|^2+\|Ex\|^2,\quad x\in\mathbb C^M,
\end{equation}
and that the spaces $\mathcal D_1$ and $\mathcal D_2$ are defect spaces associated to the isometric relation arising from \eqref{iso123123}.
Then, there exists a non-unique $\mathcal E\in\mathfrak S(\mathcal D_2,\mathcal D_1) $ such that (see \cite[(4.10), p. 237]{MR2794390}):
\begin{equation}
\label{realredheffer}
S=\Sigma_{11}+\Sigma_{12}(I_{\mathcal D_1}-\mathcal E \Sigma_{22})^{-1}\mathcal E\Sigma_{21}.
\end{equation}
Note that $\mathcal E$ is not unique, however the theorem below holds for any given choice of $\mathcal E$ satisfying \eqref{realredheffer}. In particular, the minimal norm solution
is independent of the given choice of $\mathcal E$.\\

With these preliminaries, the result of \cite{MR2794390} reads as follows, where $C_1$ and $C_2$ are as in \eqref{real123456} (see \cite[Theorem 5.1 p. 245]{MR2794390}).

\begin{theorem}
Problem \ref{pb1} has a solution if and only if the matrix $P-\mathbf {xx^*}$ is non-negative. When this condition is in force, $f$ is a solution to Problem \ref{pb1} if and only if
it is of the form
\begin{equation}
\label{sol123321}
f(z)=\Sigma_{12}(z)(I-\mathcal E(z)\Sigma_{22}(z))^{-1}
+\left((C_1+G(z)\mathcal E(z)C_2)^{-1}(I-zA)^{-1}\right)h(z)
\end{equation}
where $h$ varies in the space $\mathcal H(K_{\mathcal E})$ and has norm constraint
\[
\|h\|_{\mathcal H(K_S)}\le \sqrt{1-\|\widetilde{\mathbf x}\|^2},
\]
where $\widetilde{\mathbf x}$ is the unique vector in the range of $P$ such that
$\mathbf x=P^{1/2}\widetilde{\mathbf x}$.
\end{theorem}

In \eqref{sol123321}, the term $\Sigma_{12}(z)(I-\mathcal E(z)\Sigma_{22}(z))^{-1}$ is the minimal norm solution, and hence can also be written, in a more explicit way, using formula
\eqref{min}.\\

\begin{remark}
{\rm It would be of interest to compare this $\Sigma$ to the function $\widetilde{\Sigma}$ introduced in \cite[\S11]{ab6}
using Leech's theorem (see in the previous section equations \eqref{mult67890}-\eqref{mult6789}).
The construction in \cite{ab6} suggests connection with the inverse scattering
problem of network theory, as studied in \cite{ad1}-\cite{ad2}.}
\end{remark}

\begin{remark}{\rm The function $f$ defined by
\eqref{sol123321} {\sl a priori} is analytic only in $\mathbb D$.
In the setting of the Fock space it is clear that, {\sl a posteriori}, $f$ has an analytic extension to all of the
complex plane.}
\end{remark}
\section{The Schur algorithm for functions of the form \eqref{S=z}}
\setcounter{equation}{0}
The central objects in this paper are the functions of the form \eqref{S=z}.
In this last section we compute the Schur parameters associated to these functions. This implies to deal with the vectorial Schur algorithm, and the interested reader may find more information in, e.g., \cite{ad3}.

Let $a\in\ell_2(\mathbb N_0)$, of norm strictly less than $1$. We set $H(a)$ to be its so-called Julia extension:
\begin{equation}
H(a)=\begin{pmatrix}1&a\\a^*&I_{\ell_2(\mathbb N_0)}\end{pmatrix}
\begin{pmatrix}(1-\|a\|^2)^{-1/2}&0\\ 0&(I_{\ell_2(\mathbb N_0)}-a^*a)^{-1/2}\end{pmatrix}.
\end{equation}
We have
\[
H(a)JH(a)^*=H(a)^*JH(a)=J,\quad {\rm where}\quad J=\begin{pmatrix}1&0\\
0&-I_{\ell_2(\mathbb N_0)}\end{pmatrix}.
\]
The main result is:
\begin{theorem}
Let $\sigma_n$ be as in \eqref{sigman}, $n=0,1,\ldots$.
It holds that
\begin{equation}
\label{montmartre}
\begin{pmatrix}1-\sigma^{(0)}(z)(\sigma^{(0)}(0))^*&\sigma^{(0)}(0)-\sigma^{(0)}(z)\end{pmatrix}=z(1-|\sigma_0|^2)^{-1/2}\begin{pmatrix}0&\sigma_1&z\sigma_3&\cdots\end{pmatrix},
\end{equation}
and the Schur parameters of the function $S$ given by \eqref{S=z} are given by $\rho_0=\begin{pmatrix}0&0&0&\cdots\end{pmatrix}$ and
\begin{equation}
\rho_n=\begin{pmatrix}0&0&\cdots&0&\sqrt{\frac{1}{c_{n-1}}-\frac{1}{c_n}}
&0&\cdots&\end{pmatrix},\quad n=1,2,\ldots
\end{equation}
where the (possibly) non zero entry is on the $n$-th spot.
\end{theorem}
\begin{proof}
We consider $S(z)/z$, which we denote by $\sigma^{(0)}(z)$, and set
\[
\sigma^{(0)}(z)=\begin{pmatrix}\sigma_0& z\sigma_1&z^2\sigma_2&\cdots\end{pmatrix}.
\]
We have:
\begin{equation}
\begin{split}
\begin{pmatrix}1&-\sigma^{(0)}(z)\end{pmatrix}H(\sigma^{(0)}(0))\begin{pmatrix}z&0\\
0&I_{\ell_2(\mathbb N_0)}\end{pmatrix}&=
\begin{pmatrix}1-\sigma^{(0)}(z)(\sigma^{(0)}(0))^*&\sigma^{(0)}(0)-\sigma^{(0)}(z)\end{pmatrix}\times\\
&\hspace{-32mm}\times
\begin{pmatrix}(1-\|\sigma^{(0)}(0)\|^2)^{-1/2}&0\\
0&(I_{\ell_2(\mathbb N_0)}-(\sigma^{(0)}(0))^*\sigma^{(0)}(0))^{-1/2}\end{pmatrix}\begin{pmatrix}z&0\\
0&I_{\ell_2(\mathbb N_0)}\end{pmatrix}\\
&=\begin{pmatrix}z(1-|\sigma_0|^2)&\begin{pmatrix}0& z\sigma_1&z^2\sigma_2&\cdots\end{pmatrix}\end{pmatrix}\times\\
&\hspace{-32mm}\times
\begin{pmatrix}(1-|\sigma_0|^2)^{-1/2}&0\\
0&(I_{\ell_2(\mathbb N_0)}-(\sigma^{(0)}(0))\sigma^{(0)}(0))^{-1/2}\end{pmatrix}\\
&=z(1-|\sigma_0|^2)^{-1/2}\begin{pmatrix}0&\sigma_1&z\sigma_3&\cdots\end{pmatrix}
\end{split}
\end{equation}
since
\[
\begin{pmatrix}0& z\sigma_1&z^2\sigma_2&\cdots\end{pmatrix}=\begin{pmatrix}0& z\sigma_1&z^2\sigma_2&\cdots\end{pmatrix}(I_{\ell_2(\mathbb N_0)}-(\sigma^{(0)}(0))^*\sigma^{(0)}(0))^{-1/2}.
\]
This gives \eqref{montmartre} and the formula for $\rho_1$. The whole procedure can be iterated and so the case $n>1$ is treated in the same way.
\end{proof}

\end{document}